\documentclass[11pt,reqno,oneside]{amsart}
\usepackage{amsmath}
\usepackage{amssymb}
\usepackage{amstext}
\usepackage{amsfonts}
\usepackage{amsthm}
\usepackage{ifthen}
\usepackage{xspace}
\usepackage{comment}
\usepackage{graphicx}
\usepackage{mathrsfs}
\usepackage{enumerate}

\numberwithin{equation}{section}  

\newtheorem{punkt}{}[section]

\theoremstyle{plain}

\newtheorem{lemma}[punkt]{Lemma}
\newtheorem{proposition}[punkt]{Proposition}
\newtheorem{theorem}[punkt]{Theorem}

\theoremstyle{definition}

\theoremstyle{plain}
\newtheorem*{corollary*}{Corollary}
\newtheorem*{lemma*}{Lemma}
\newtheorem*{proposition*}{Proposition}
\newtheorem*{theorem*}{Theorem}

\theoremstyle{definition}
\newtheorem*{remark*}{Remark}
\newtheorem*{remarks*}{Remarks}
\newtheorem*{example*}{Example}
\newtheorem*{examples*}{Examples}
\newtheorem*{definition*}{Definition}
\newtheorem*{conjecture*}{Conjecture}
\newtheorem*{assumption*}{Assumption}
\newtheorem*{assumptions*}{Assumptions}
\newtheorem*{construction*}{Construction}

\def\mylebesgue{\lambda \mskip -8mu \lambda}

\def\myo{\mathcal{O}}

\def\ee{\mathbb{E}}

\def\re{\qopname\relax{no}{Re}\,}
\def\im{\qopname\relax{no}{Im}\,}

\def\ai{\qopname\relax{no}{Ai}\nolimits}

\def\eg{e.g.\@\xspace}
\def\ie{i.e.\@\xspace}
\def\iid{i.i.d.\@\xspace}

\begin{document}

\title{Characteristic Polynomials \\[+5pt] of Sample Covariance Matrices: \\[+5pt] The Non-Square Case}

\author{H. K\"osters}
\address{Holger K\"osters, Fakult\"at f\"ur Mathematik, Universit\"at Bielefeld,
Postfach 100131, 33501 Bielefeld, Germany}
\email{hkoesters@math.uni-bielefeld.de}


\begin{abstract}
We consider the sample covariance matrices of large data matrices
which have \iid complex matrix entries and which are 
non-square in the sense that the difference between 
the number of rows and the number of columns tends to infinity.
We show that the second-order correlation function 
of the characteristic polynomial of the sample covariance matrix
is asymptotically given by the sine kernel in the bulk of the spectrum
and by the Airy kernel at the edge of the spectrum.
Similar results are given for real sample covariance matrices.
\end{abstract}

\maketitle

\markboth{H. K\"osters}{Characteristic Polynomials of Sample Covariance Matrices}

\bigskip

\section{Introduction}

The characteristic polynomials of random matrices have 
attracted considerable attention in the last few years,
one reason being that their correlations seem to reflect
the correlations of the eigenvalues
\cite{BH1,BH2,BH3,FS3,BDS,SF3,AF3,BS,Va,GK,Ko1,Ko2,Ko3}.

In this paper we continue with our investigation \cite{Ko3} 
of the second-order correlation function 
of the characteristic polynomial of a sample covariance matrix.
This~function is given by
\begin{align}
\label{cf-def}
f(n,m;\mu,\nu) := \ee \left( \det(Z(n,m)-\mu) \det(Z(n,m)-\nu) \right) \,,
\end{align}
where $n,m \in \mathbb{N}$, $\mu,\nu \in \mathbb{R}$,
and $Z(n,m)$ is a complex or real sample covariance matrix
defined as follows:

\medskip

\textbf{Complex Sample Covariance Matrices.}
Let $Q$ be a distribution on the real line  
with expectation $0$, variance $1/2$ and finite fourth moment $b$,
and for given $n,m \in \mathbb{N}$ with $n \geq m$,
let $X := X(n,m) := (X_{ij})_{i=1,\hdots,n;j=1,\hdots,m}$
denote the $n \times m$ matrix whose entries $X_{ij}$ 
are \iid complex random variables 
whose real and imaginary parts are independent,
each with distribution $Q$.
Let $X^* = X^*(n,m)$ denote the conjugate transpose of $Z$.
Then the Hermitian $m \times m$ matrix $Z := Z(n,m) := X(n,m)^* X(n,m)$
is called the (unrescaled) \emph{sample covariance matrix} associated 
with the distribution $Q$. 

\medskip

\textbf{Real Sample Covariance Matrices.}
Let $Q$ be a distribution on the real line
with ex\-pectation $0$, variance $1$ and finite fourth moment $b$,
and for given $n,m \in \mathbb{N}$ with $n \geq m$,
let $X := X(n,m) := (X_{ij})_{i=1,\hdots,n;j=1,\hdots,m}$
denote the $n \times m$ matrix whose entries $X_{ij}$ 
are \iid real random variables with distribution $Q$.
Let $X^T = X^T(n,m)$ denote the transpose of $X$.
Then the symmetric $m \times m$ matrix $Z := Z(n,m) := X(n,m)^T X(n,m)$
is called the (unrescaled) \emph{sample covariance matrix} associated 
with the distribution $Q$. 

\pagebreak[2]
\medskip

We are interested in the asymptotic behavior of the~values
$f(n_N,m_N;\mu_N,\nu_N)$ as $N \to \infty$,
for certain choices of the sequences $(n_N)$, $(m_N)$, $(\mu_N)$, $(\nu_N)$.

In a recent paper \cite{Ko3} we considered the ``square'' case
where the difference $\alpha_N := n_N - m_N$ is fixed.
We showed that in this situation, the second-order~correlation function 
of the characteristic polynomial of a complex sample covariance matrix 
is asymptotically given
(after the appropriate rescaling)
\newline -- by the sine kernel in the bulk of the spectrum,
\newline -- by the Airy kernel at the soft edge of the spectrum,
\newline -- by the Bessel kernel at the hard edge of the spectrum.
\newline Moreover, similar results were obtained for real sample covariance matrices.

The purpose of this note is to derive similar results for the ``non-square'' case
where the difference $\alpha_N := n_N - m_N$ tends to infinity 
in such a way that the ratio $\gamma_N := m_N / n_N$ 
tends to some constant $\gamma_\infty \in {(}0;1{)}$ sufficiently quickly. 
We will show that in this situation, the second-order correlation function 
of the characteristic polynomial of a complex sample covariance matrix 
is asymptotically given
(after the appropriate rescaling)
\newline -- by the sine kernel in the bulk of the spectrum,
\newline -- by the Airy kernel at the edge of the spectrum.
\newline Note that both edges of the spectrum are ``soft'' 
when $\gamma_\infty \in {(}0;1{)}$.

These results are in consistence with results
by \textsc{Ben Arous} and \textsc{P\'ech\'e} \cite{BP},
\textsc{Tao} and \textsc{Vu} \cite{TV5},
\textsc{Soshnikov} \cite{So}, \textsc{P\'ech\'e} \cite{Pe},
and \textsc{Feldheim} and \textsc{Sodin} \cite{FeSo}, 
who obtained similar results for the (more relevant) correlation function
of the eigenvalues, yet under stronger assumptions on the underlying distributions.
\textsc{Ben~Arous} and \textsc{P\'ech\'e} \cite{BP}
proved the occurrence of the sine kernel in the~bulk of the spectrum
for a certain class of complex sample covariance matrices
with $\gamma_\infty \!=\! 1$.
Very recently, \textsc{Tao} and \textsc{Vu} \cite{TV5} 
extended this result to a quite general class of sample covariance matrices,
still with $\gamma_\infty = 1$.
\textsc{Soshnikov} \cite{So} and \textsc{P\'ech\'e}~\cite{Pe}
established the occurrence of the Airy kernel at the upper edge of the spectrum
for real and complex sample covariance matrices 
whose underlying distributions are symmetric with exponential tails
(or at least finite 36th moments).
Recently, \textsc{Feldheim} and \textsc{Sodin} \cite{FeSo}
proposed another approach to obtain these results,
which also works for the lower edge of the spectrum.
Thus, our results add some support to the wide-spread expectation
that correlation functions in random matrix theory are universal,
and subject to very weak moment conditions only.

It is well-known that the asymptotic distribution of the global spectrum
of the rescaled sample covariance matrix $(1/n) \, Z(n,m)$ is given 
(both in the complex case and in the real case) 
by the Mar\v{c}enko-Pastur distribution of parameter $\gamma_\infty$, 
which has the density
\begin{align}
\label{mp}
g(\xi) := \frac{1}{2\pi\gamma_\infty \, \xi} \sqrt{(\xi-\xi_*)(\xi^*-\xi)} \,, \quad \xi_* \leq \xi \leq \xi^* \,,
\end{align}
where $\xi_* := (1-\sqrt{\gamma_\infty})^2$ and $\xi^* := (1+\sqrt{\gamma_\infty})^2$.
Note that $\xi_*,\xi^*$ as well as $g$ implicitly depend on $\gamma_\infty$,
but this dependence will be kept implicit throughout this paper.
Also, note that $\xi_*,\xi^*$ are the solutions to the equation
\begin{align}
\xi^2 - 2 (1+\gamma_\infty) \xi + (1-\gamma_\infty)^2 = 0 \,.
\label{boundary}
\end{align}

Unless otherwise mentioned, we will always assume that
$1 \leq m_N \leq n_N := N$ for all $N \in \mathbb{N}$,
$\alpha_N := n_N - m_N \to \infty$ as~$N \to \infty$, and 
$\gamma_N := m_N / n_N \to \gamma_\infty$ as~$N \to \infty$,
where $\gamma_\infty \in {]}0;1{]}$.
Our main results are as~follows:

\begin{theorem}
\label{bulk-1}
Suppose that $n_N = N$ and $m_N = N \gamma_\infty + o(N)$ as $N \to \infty$.
Let $f$ denote the second-order correlation function of a complex sample covariance matrix
satisfying our standing moment conditions.
For any $\xi \in (\xi_*,\xi^*)$, $\mu,\nu \in \mathbb{R}$, setting
$$
Z_N(\xi,\mu,\nu) := \left( N^2 \xi^2 + N \xi (\mu + \nu) + \mu \nu \right)^{\alpha_N/2} \, \exp \left( -N\xi - \tfrac12(\mu+\nu) \right) \,,
$$
$\hat\mu := \mu / (\gamma_\infty g(\xi))$,
$\hat\nu := \nu / (\gamma_\infty g(\xi))$,
we have
\begin{multline}
\label{bulk-1F}
\lim_{N \to \infty} \left( (\gamma_\infty g(\xi))^{-1} \, Z_N(\xi,\hat\mu,\hat\nu) \cdot \frac{f(n,m;N\xi+\hat\mu,N\xi+\hat\nu)}{n! \, m!} \right) \\
= \exp(b^*) \, \mathbb{S}(\mu,\nu) \,,
\end{multline}
where $b^* := 2(b-\tfrac34)$ and $\mathbb{S}$ is defined by
\begin{align}
\label{sine}
\mathbb{S}(x,y) := \frac{\sin \pi(x-y)}{\pi(x-y)} \,.
\end{align}
\end{theorem}

\begin{theorem}
\label{bulk-2}
Suppose that $n_N = N$ and $m_N = N \gamma_\infty + o(N)$ as $N \to \infty$.
Let $f$ denote the second-order correlation function of a real sample covariance matrix
satisfying our standing moment conditions.
For any $\xi \in (\xi_*,\xi^*)$, $\mu,\nu \in \mathbb{R}$, setting
$$
Z_N(\xi,\mu,\nu) := \left( N^2 \xi^2 + N \xi (\mu + \nu) + \mu \nu \right)^{\alpha_N/2} \, \exp \left( -N\xi - \tfrac12(\mu+\nu) \right) \,,
$$
$\hat\mu := \mu / (\gamma_\infty g(\xi))$,
$\hat\nu := \nu / (\gamma_\infty g(\xi))$,
we have
\begin{multline}
\label{bulk-2F}
\lim_{N \to \infty} \left( N^{-1} \, \xi^{-1} \, (\gamma_\infty g(\xi))^{-3} \, Z_N(\xi,\hat\mu,\hat\nu) \cdot \frac{f(n,m;N\xi+\hat\mu,N\xi+\hat\nu)}{n! \, m!} \right) \\
= \exp(b^*) \, \widetilde{\mathbb{S}}(\mu,\nu) \,,
\end{multline}
where $b^* := (b-3)$ and $\widetilde{\mathbb{S}}$ is the sine kernel defined by
\begin{align}
\label{sine2}
\widetilde{\mathbb{S}}(x,y) := \frac{2\sin \pi(x-y)}{\pi(x-y)^3} - \frac{2\cos \pi(x-y)}{(x-y)^2} \,.
\end{align}
\end{theorem}

\begin{theorem}
\label{edge-1}
Suppose that $n_N = N$ and $m_N = N \gamma_\infty + o(N^{1/3})$ as $N \to \infty$.
Let $f$ denote the second-order correlation function of a complex sample covariance matrix
satisfying our standing moment conditions.
For any $\mu,\nu \in \mathbb{R}$, setting
$$
Z_N(\xi,\mu,\nu) := \left( \xi^2 N^2 + (\mu + \nu) \xi N^{4/3}  + \mu \nu N^{2/3} \right)^{\alpha_N/2} \, \exp \left( - \xi N - \tfrac12(\mu+\nu)N^{1/3} \right) \,,
$$
$\hat\mu := \xi^{2/3} \gamma_\infty^{-1/6} \mu$,
$\hat\nu := \xi^{2/3} \gamma_\infty^{-1/6} \nu$,
we have
\begin{multline}
\label{edge-1F}
\lim_{N \to \infty} \left( \xi^{2/3} \gamma_\infty^{-1/6} \, N^{1/3} \, Z_N(\xi,+\hat\mu,+\hat\nu) \cdot \frac{f(n,m;\xi N+\hat\mu N^{1/3},\xi N+\hat\nu N^{1/3})}{n! \, m!} \right) \\
= \exp(b^*) \, \mathbb{A}(\mu,\nu) 
\end{multline}
for $\xi = \xi^*$ and, when $\gamma_\infty \ne 1$,
\begin{multline}
\label{edge-1G}
\lim_{N \to \infty} \left( \xi^{2/3} \gamma_\infty^{-1/6} \, N^{1/3} \, Z_N(\xi,-\hat\mu,-\hat\nu) \cdot \frac{f(n,m;\xi N-\hat\mu N^{1/3},\xi N-\hat\nu N^{1/3})}{n! \, m!} \right) \\
= \exp(b^*) \, \mathbb{A}(\mu,\nu) 
\end{multline}
for $\xi = \xi_*$,
where $b^* := 2(b-\tfrac34)$ and $\mathbb{A}$ is the Airy kernel defined by
\begin{align}
\label{airy}
\mathbb{A}(x,y) := \frac{\ai(x) \ai'(y) - \ai'(x) \ai(y)}{x-y} \,.
\end{align}
\end{theorem}

\begin{theorem}
\label{edge-2}
Suppose that $n_N = N$ and $m_N = N \gamma_\infty + o(N^{1/3})$ as $N \to \infty$.
Let $f$ denote the second-order correlation function of a real sample covariance matrix
satisfying our standing moment conditions.
For any $\mu,\nu \in \mathbb{R}$, setting
$$
Z_N(\xi,\mu,\nu) := \left( \xi^2 N^2 + (\mu + \nu) \xi N^{4/3}  + \mu \nu N^{2/3} \right)^{\alpha_N/2} \, \exp \left( - \xi N - \tfrac12(\mu+\nu)N^{1/3} \right) \,,
$$
$\hat\mu := \xi^{2/3} \gamma_\infty^{-1/6} \mu$,
$\hat\nu := \xi^{2/3} \gamma_\infty^{-1/6} \nu$,
we have
\begin{multline}
\label{edge-2F}
\lim_{N \to \infty} \left( \xi \, \gamma_\infty^{-1/2} \, Z_N(\xi,+\hat\mu,+\hat\nu) \cdot \frac{f(n,m;\xi N+\hat\mu N^{1/3},\xi N+\hat\nu N^{1/3})}{n! \, m!} \right) \\
= \exp(b^*) \, \widetilde{\mathbb{A}}(\mu,\nu) 
\end{multline}
for $\xi = \xi^*$ and, when $\gamma_\infty \ne 1$,
\begin{multline}
\label{edge-2G}
\lim_{N \to \infty} \left( \xi \, \gamma_\infty^{-1/2} \, Z_N(\xi,-\hat\mu,-\hat\nu) \cdot \frac{f(n,m;\xi N-\hat\mu N^{1/3},\xi N-\hat\nu N^{1/3})}{n! \, m!} \right) \\
= \exp(b^*) \, \widetilde{\mathbb{A}}(\mu,\nu) 
\end{multline}
for $\xi = \xi_*$,
where $b^* := (b-3)$ and $\widetilde{\mathbb{A}}$ is defined by
\begin{multline}
\label{airy2}
\qquad \widetilde{\mathbb{A}}(x,y) := \frac{2 \ai(x) \ai'(y) - 2 \ai'(x) \ai(y)}{(x-y)^3} \\ + \frac{(x+y) \ai(x) \ai(y) - 2 \ai'(x) \ai'(y)}{(x-y)^2} \,. \quad
\end{multline}
\end{theorem}

Note that in all the cases, the limit is almost universal 
in that it depends on the underlying distribution only 
via a multiplicative factor containing the fourth cumulant.
Also, note that we obtain the sine kernel in the complex setting
not~only for $\gamma_\infty = 1$ but also for $\gamma_\infty < 1$.

This paper is organized as follows. Section~2 contains
the proofs of Theorems \ref{bulk-1} and \ref{edge-1}.
(Theorems \ref{bulk-2} and \ref{edge-2} may be proved similarly.)
In Section~3 we~state some auxiliary results concerning Bessel functions
which are needed for the proofs.

\bigskip

\section{The Proofs of Theorems \ref{bulk-1} and \ref{edge-1}}

The starting point for the proofs of Theorems \ref{bulk-1} and \ref{edge-1}
will be the following exponential-type generating function
for the second-order correlation function of the characteristic polynomial
of the (unrescaled) sample covariance matrix
satisfying the moment conditions from Section~1:

\begin{proposition}
\label{GF}
For any $\alpha \in \mathbb{N}$, $|z| < 1$,
\begin{align}
\sum_{m=0}^{\infty} \frac{f(m+\alpha,m;\mu,\nu)}{(m+\alpha)! \, m!} \, z^m
=
\frac{\exp \left( -(\mu\!+\!\nu) \frac{z}{1-z} + b^* z \right) \cdot I_\alpha \left( \frac{2 (\mu \nu z)^{1/2}}{1-z} \right)}{(1-z)^{1+(2/\beta)} \cdot ((\mu \nu z)^{1/2})^\alpha} \,,
\label{GFF}
\end{align}
where 
$b^* := 2(b-\tfrac34)$ and $\beta := 2$ in the complex setting,
$b^* := b-3$ and $\beta := 1$ in the real setting,
and $I_\alpha$ denotes the modified Bessel function of order $\alpha$.
\end{proposition}

Note that $(w/2)^{-\alpha} I_\alpha(w)$ is an even entire function of $w$, as follows 
directly from the power series representation of the modified Bessel function.
Thus, the right-hand side in \eqref{GFF} extends to a meromorphic function 
on the complex plane, with a singularity at $z=1$. In the subsequent calculations, 
we will always consider the principal branch of $I_\alpha(w)$.

We only give the proofs for the complex setting,
the proofs for the real setting being very similar.

We will always assume that $\alpha_N := n_N - m_N \to \infty$
and $\gamma_N := m_N / n_N \to \gamma_\infty$
for~some constant $\gamma_\infty \in {]}0,1{]}$.
(The case where $\alpha_N$ remains bounded is already con\-sidered in \cite{Ko3}.)
Furthermore, we will usually omit the index $N$ of the parameters,
and also of several auxiliary functions defined below.
By Proposition \ref{GF}, we have the~integral representation
\begin{align}
\label{intrep}
\frac{f(n,m;\mu,\nu)}{n! \, m!}
=
\frac{1}{2 \pi i} \int_\sigma
\frac{\exp \left( -(\mu+\nu) \frac{z}{1-z} + b^* z \right) \cdot I_\alpha \left( \frac{2 (\mu \nu z)^{1/2}}{1-z} \right)}{(1-z)^2 \cdot ((\mu \nu z)^{1/2})^\alpha}
 \ \frac{dz}{z^{m+1}} \,,
\end{align}
where $\sigma$ denotes a contour around the origin.
Similarly as in \cite{Ko3}, we will show that 
the main contribution to the contour integral in \eqref{intrep} 
comes from a small neighborhood of the point $z_0 = 1$.
To determine the asymptotic behavior of the values $f(n_N,m_N,\mu_N,\nu_N)$
in Theorems \ref{bulk-1} and \ref{edge-1}, we proceed similarly as in \cite{Ko3},
but we have to address some additional complications arising from the circumstance
that $\alpha_N \to \infty$.
To deal with the this problem, we will use various well-known 
\emph{uniform asymptotics} for the modified Bessel function
of large positive order (see \eg Chapters 10~and~11 in \textsc{Olver} \cite{O}).
For convenience, we state these results (and several bounds deduced therefrom)
in Section~3 of this paper.

We will choose the contour $\sigma = \sigma_N$ consisting 
of the following parts $\sigma_i = \sigma_{N,i}$, $i=-2,\hdots,+2$
(see~Figure~\ref{newcontour}):
\begin{itemize}
\item $\sigma_{-2}(t) := \exp(it)$, $-\pi \leq t \leq -ia/N^{\eta}$,
\item $\sigma_{-1}(t) := (1-t/N^{\eta}) \exp(-ia/N^{\eta})$, $0 \leq t \leq 1$,
\item $\sigma_{ 0}(t) := (1-1/N^{\eta}) \exp(it)$, $-ia/N^{\eta} \leq t \leq ia/N^{\eta}$,
\item $\sigma_{+1}(t) := (1+t/N^{\eta}) \exp(ia/N^{\eta})$, $-1 \leq t \leq 0$,
\item $\sigma_{+2}(t) := \exp(it)$, $ia/N^{\eta} \leq t \leq \pi$,
\end{itemize} 
Here $a > 0$ is a positive real number which will finally be chosen sufficiently large,
and $\eta := 1$ for the bulk of the spectrum and $\eta := 1/3$ for the edge of the spectrum.

\begin{figure}
\begin{center}
\includegraphics[width=6.6cm]{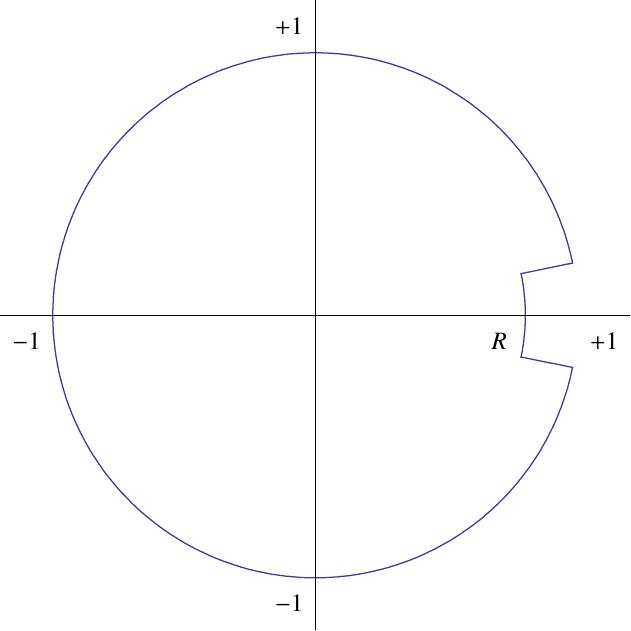}
\end{center}
\caption{The contour $\sigma$.}
\label{newcontour}
\end{figure}

The rough idea behind the choice of this contour is the following:
After the appropriate rescaling, the main contribution
to the integral \eqref{intrep} comes from the part $\sigma_0$,
for which we can use the uniform asymptotics \eqref{bessel-u1}
for the modified Bessel function in the positive half-plane.
The other parts are essentially error terms, 
which can be made arbitrary small by picking $a$ and $N$ sufficiently large.
However, to prove this, we need rather precise approximations and bounds
for the modified Bessel function close to the imaginary axis,
and this is our main motivation for the choice of the above contour.
For the contours $\sigma_{-1}$ and $\sigma_{+1}$,
we can use the uniform bound \eqref{bessel-u2}
for the modified Bessel function \emph{near} the imaginary axis,
but away from the turning points $\pm i$.
For the contours $\sigma_{-2}$ and $\sigma_{+2}$,
we can use the uniform bounds \eqref{jbound-1} and \eqref{jbound-2}
for the modified Bessel function \emph{on} the imaginary axis,
including the vicinity of the turning points $\pm i$.

In the following proofs, we adopt the convention that all asymptotic bounds 
may depend on the sequence $(\alpha_N)_{N \in \mathbb{N}}$ as well as 
on the ``shift parameters'' $\xi,\mu,\nu$, unless otherwise indicated.
A similar remark applies to large positive constants,
typically denoted by $C$.


After these comments, we turn to the proofs of the main theorems.

\begin{proof}[Proof of Theorem \ref{bulk-1}]
Let $\xi,\mu,\nu$ be as in Theorem \ref{bulk-1}, and let $a > 1$.
We~assume throughout this proof that $N$ is sufficiently large.
By (\ref{intrep}) and the definition of $Z_N(\xi,\mu,\nu)$, we have
\begin{align}
\label{intrep11}
Z_N(\xi,\mu,\nu) \cdot \frac{f(n,m;N\xi+\mu,N\xi+\nu)}{n! \, m!} =
\frac{1}{2 \pi i} \int_\sigma
h(z)
\ \frac{dz}{z} \,, 
\end{align}
where
$$
h(z) = h_N(z) := \frac{\exp \left( -(N\xi+\tfrac12(\mu+\nu)) \frac{1+z}{1-z} + b^*z \right) \cdot I_\alpha \left( w(z) \right)}{(1-z)^2 \cdot z^{(n+m)/2}} \,,
$$
$$
w(z) = w_N(z) := 2 \sqrt{N^2 \xi^2 + N \xi (\mu + \nu) + \mu \nu} \cdot \frac{\sqrt{z}}{1-z} \,,
$$
and $\sigma = \sigma_N$ is the contour specified above, with $\eta := 1$.
Note that $\sigma$ implicitly depends on $a$.

\pagebreak[2]

For any $a \in [1,\infty]$, let
$$
S(a) := \frac{\exp(b^*)}{4\pi^{3/2}\xi^{1/2}} \int_{-a}^{+a} \frac{\exp \left( h_\infty(\xi) (1-iu) - \tfrac14 (\mu-\nu)^2 \xi^{-1} (1-iu)^{-1} \right)}{(1-iu)^{3/2}} \, du \,,
$$
where $h_\infty(\xi) := -\tfrac14 \xi + \tfrac12 (1+\gamma_\infty) - \tfrac14 (1-\gamma_\infty)^2 \xi^{-1}$.
Note that by~\eqref{boundary} we have $h_\infty(\xi) > 0$ for $\xi \in [\xi_*,\xi^*]$.
We will show that for any $\delta > 0$, there exists some $a > 1$ 
such that the following holds:
\begin{align}
\label{claim11}
\lim_{N \to \infty}
\frac{1}{2 \pi i} \int_{\sigma_{0}} h(z) \ \frac{dz}{z} = S(a) \,,
\end{align}
\begin{align}
\label{claim12}
\limsup_{N \to \infty}
\left| \int_{\sigma_{1}} h(z) \ \frac{dz}{z} \right| \leq \delta \,,
\end{align}
\begin{align}
\label{claim13}
\limsup_{N \to \infty}
\left| \int_{\sigma_{2}} h(z) \ \frac{dz}{z} \right| \leq \delta \,.
\end{align}
Then, by symmetry, similar bounds hold for the integrals along $\sigma_{-1}$ and $\sigma_{-2}$.
Using these results, it is easy to see that
$$
\lim_{N \to \infty}
\frac{1}{2 \pi i} \int_{\sigma} h(z) \ \frac{dz}{z} = S(\infty) \,.
$$
Now, by Laplace inversion, we have, for $t > 0$, $a \in \mathbb{R}$,
$$
\frac{1}{2\pi i} \int_{1-i\infty}^{1+i\infty} e^{tz} \, \frac{e^{-a^2/4z}}{z^{3/2}} \ dz 
= 
\frac{2 \sin (a \sqrt{t})}{\sqrt{\pi} a}
$$
(see \eg p.\,245 in \cite{BMP}).
We therefore get
$$
S(\infty) = \tfrac{1}{\pi} \, \exp(b^*) \, \left( h_\infty(\xi) / \xi \right)^{1/2} \cdot \frac{\sin \left( (\mu-\nu) \left( h_\infty(\xi)/\xi \right)^{1/2} \right)}{\left( (\mu-\nu) \left( h_\infty(\xi)/\xi \right)^{1/2} \right)} \,.
$$
Now replace the local shift parameters $\mu,\nu$ 
with $\mu/(\gamma_\infty g(\xi)),\nu/(\gamma_\infty g(\xi))$ and \linebreak
observe that $\left( h_\infty(\xi)/\xi \right)^{1/2} = \pi \gamma_\infty g(\xi)$
(which is a simple consequence of \eqref{boundary}) to obtain \eqref{bulk-1F}.

\pagebreak[2]
\medskip

To prove \eqref{claim11} -- \eqref{claim13}, we write
$$
h(z) = \frac{h_1(z) h_2(z)}{(1-z)^{2} \, z^{(n+m)/2}} \,,
$$
where
$$
h_1(z) := \exp \left( -(N\xi + \tfrac12(\mu+\nu)) \frac{1+z}{1-z} + b^*z \right) \cdot \exp \left( + w(z) \right)
$$
and
$$
h_2(z) := \exp \left( - w(z) \right) \cdot I_\alpha \left( w(z) \right) \,.
$$

\pagebreak[2]

\noindent{}Since
$$
w(z) = 2 \left( N \xi + \tfrac12 (\mu + \nu) - \tfrac18 (\mu - \nu)^2 / N \xi + \mathcal{O}(1/N^2) \right) \frac{\sqrt{z}}{1-z} \,,
$$
a simple calculation yields
\begin{multline*}
h_1(z) = \exp \bigg( -N\xi \frac{1-\sqrt{z}}{1+\sqrt{z}} - \tfrac12(\mu+\nu) \frac{1-\sqrt{z}}{1+\sqrt{z}} - \tfrac14 (\mu-\nu)^2 / (N\xi) \, \frac{\sqrt{z}}{1-z} + b^*z \\ \,+\, \frac{\sqrt{z}}{1-z} \, \myo \big( 1/N^2 \big) \bigg) \,.
\end{multline*}
(Recall our convention that implicit constants in $\myo$-terms 
may depend on $(\alpha_N)$, $\xi,\mu,\nu$, which are regarded as fixed).

\medskip

We first prove \eqref{claim11}. In doing so, we use the notation $\myo_a$ 
to denote a bound involving an implicit constant depending also on $a$ 
(in addition to $(\alpha_N)$, $\xi,\mu,\nu$). Let $R = R_N := (1-1/N)$.
Substituting $z = Re^{it}$ and $t = u/N$ in the integral 
on the left-hand side in \eqref{claim11}, we obtain
\begin{align}
\label{step11a}
\frac{1}{2 \pi N} \int_{-a}^{+a}
\frac{h_1(Re^{iu/N}) \, h_2(Re^{iu/N})}{(1-Re^{iu/N})^2 \cdot (Re^{iu/N})^{(n+m)/2}}
 \ du \,.
\end{align}
Now, using Taylor expansion, we have the following approximations, 
for $|u| \leq a$:
\begin{align*}
h_1(Re^{iu/N}) &= \exp \big( -\tfrac14\xi(1-iu) - \tfrac14(\mu-\nu)^2 \xi^{-1} / (1-iu) + b^* \big) \left( 1 + o_a(1) \right) \,,
\\
h_2(Re^{iu/N}) &= \frac{\exp\left( - \tfrac14 (1-\gamma_\infty)^2 \xi^{-1} (1-iu) \right) }{\sqrt{4\pi N^2 \xi / (1-iu) }} (1 + o_a(1)) \,,
\\
(1-Re^{iu/N})^2 &= (1-iu)^2/N^2 \left( 1 + o_a(1) \right) \,,
\\
(Re^{iu/N})^{(m+n)/2} &= \exp \big( -\tfrac12 (1-iu) (1+\gamma_\infty) \big) \left( 1 + o_a(1) \right) \,.
\end{align*}
For the second approximation, we have also used the uniform asymptotics \eqref{bessel-u1}
for the modified Bessel function. Indeed, since
$$
\alpha_N = \left( 1 - \gamma_\infty + o(1) \right) N \,,
$$
$$
w(Re^{iu/N}) = 2N^2 \xi / (1-iu) \left( 1 + o_a(1) \right) \,,
$$
and
\begin{align}
\label{etaex}
\sqrt{1+z^2} + \log \frac{z}{1+\sqrt{1+z^2}} = z - \tfrac{1}{2} z^{-1} + \myo(z^{-3})
\end{align}
for $|z| \to \infty$, $|\arg z| \leq \tfrac12\pi$, it follows from \eqref{bessel-u1} that
\begin{align*}
   h_2(Re^{iu/N})   
&= \exp(-w(Re^{iu/N})) \cdot I_\alpha(w(Re^{iu/N})) \\
&= \frac{\exp \left( - \tfrac12\alpha^2/w(Re^{iu/N}) \right)}{\sqrt{2\pi w(Re^{iu/N})}} \left( 1 + o_a(1) \right) \,,
\end{align*}
whence the approximation for $h_2(Re^{iu/N})$.
Inserting the above approximations into~\eqref{step11a} 
and recalling the definition of $S(a)$ yields \eqref{claim11}.

\medskip

We now prove \eqref{claim12} and \eqref{claim13}. 
For $z \in \text{range}\,(\sigma_1) \cup \text{range}\,(\sigma_2)$, 
write $\sqrt{z} = re^{i\varphi}$ with $r \in [(1-1/N)^{1/2};1]$, $\varphi \in [a/2N;\pi/2]$.
Since $a > 1$, we then have the estimates 
\begin{align}
\label{bound1a}
  \left| \re\!\left( \frac{1-\sqrt{z}}{1+\sqrt{z}} \right) \right|
= \frac{1-r^2}{1+2r\cos\varphi+r^2}
\leq 1/N \,,
\end{align}
\begin{align}
\label{bound1b}
  \left| \frac{\sqrt{z}}{1-z} \right|
= \frac{r}{|1-r^2e^{2i\varphi}|}
\leq \frac{1}{|r^2-r^2e^{2i\varphi}|}
= \myo(N) \,,
\end{align}
\begin{align}
\label{bound1c}
  \re \left( \frac{\sqrt{z}}{1-z} \right)
= \frac{(r-r^3)\cos\varphi}{1-2r^2\cos2\varphi+r^4}
\geq 0 \,,
\end{align}
\begin{align}
\label{bound1d}
  \im \left( \frac{\sqrt{z}}{1-z} \right)
= \frac{(r+r^3)\sin\varphi}{1-2r^2\cos2\varphi+r^4}
\geq 0 \,.
\end{align}
In particular, \eqref{bound1a} and \eqref{bound1b} show that $h_1(z)$ is uniformly bounded.
Moreover,
$$
|z^{-(n+m)/2}| \leq (1+2/N)^{(n+m)/2} = (1+2/N)^{\frac12(1+\gamma_\infty+o(1))N} = \myo(1) \,.
$$
The proof of \eqref{claim12} and \eqref{claim13} is therefore reduced to showing
that for any $\delta > 0$, there exists some $a > 1$ such that
\begin{align}
\label{claim14}
\limsup_{N \to \infty}
\left| \int_{\sigma_{1}} \frac{h_2(z)}{(1-z)^2} \ \frac{dz}{z} \right| \leq \delta \,,
\qquad
\limsup_{N \to \infty}
\left| \int_{\sigma_{2}} \frac{h_2(z)}{(1-z)^2} \ \frac{dz}{z} \right| \leq \delta \,.
\end{align}

For the first integral in \eqref{claim14}, 
note that for $t \in [0;1]$ and fixed $a > 1$, 
we have 
$$
\alpha_N = \left( 1 - \gamma_\infty + o(1) \right) N \,,
$$
and
$$
w \left( (1-t/N) \, e^{ia/N} \right) = 2N^2 \xi / (t-ia) \left( 1 + o_a(1) \right) \,.
$$
Writing $z = (1-t/N) \, e^{ia/N}$ with $t \in [0;1]$ and using \eqref{etaex}, \eqref{bound1c} 
as well as \linebreak the uniform bound \eqref{bessel-u3} for the modified Bessel function, 
it therefore follows that for all sufficiently large $N \in \mathbb{N}$ (depending on $a$),
\begin{align*}
   \left| h_2(z) \right|  
 = \left| \exp(-w(z)) \cdot I_\alpha(w(z)) \right|
&\leq \frac{C \left| \exp \left( - \tfrac12\alpha^2/w(z) \right) \right|}{\sqrt{2\pi \left| w(z) \right|}} \left( 1 + o_a(1) \right) \\
&\leq \frac{C (1+a^2)^{1/4}}{\sqrt{4 \pi N^2 \xi}} \left( 1 + o_a(1) \right) \,,
\end{align*}
where $C$ is a positive constant which does not depend on $a$
and which may change from step to step in the sequel.
Combining this with the estimates
$|1-z| \geq a/N (1+o_a(1))$, $|z| \geq 1-1/N$,
it follows that
$$
\limsup_{N \to \infty} \left| \int_{\sigma_{1}} \frac{h_2(z)}{(1-z)^2} \ \frac{dz}{z} \right| \leq \frac{C (1+a^2)^{1/4}}{\sqrt{4\pi\xi} \, a^2} \,.
$$
Since for $a > 1$ sufficiently large, the right-hand side is clearly
bounded above by~$\delta$, this proves \eqref{claim12}.

\pagebreak[2]

For the second integral in \eqref{claim14}, note that for $|z| = 1$, 
we have $\re w(z) = 0$ and $\im w(z) \geq 0$ 
from \eqref{bound1c} and \eqref{bound1d}.
Thus, we may use the bounds \eqref{jbound-1a}~and~\eqref{jbound-2a}
for the modified Bessel function on the imaginary axis.
\pagebreak[2]
Writing $z = e^{it}$ with $t \in [1/N;\pi]$, 
we have, for all sufficiently large $N \in \mathbb{N}$,
$$
\big| w_N(z) \big| = 2 \sqrt{N^2\xi^2+N\xi(\mu+\nu)+\mu\nu} \left| \frac{\sqrt{z}}{1-z} \right| \ge \frac{N\xi}{|1-e^{it}|}
$$
and
$$
\big| w_N(z) / \alpha_N \big| \ge \frac{\xi}{(1-\gamma_N)|1-e^{it}|} \geq \frac{\xi}{|1-e^{it}|} \geq \frac{\xi}{t} \,.
$$
Hence, by \eqref{jbound-1a} and \eqref{jbound-2a},
for all sufficiently large $N \in \mathbb{N}$,
$$
|h_2(e^{it})| 
\leq 
\begin{cases}
C \, |1-e^{it}|^{1/2} \,/\, N^{1/2} & ; \ |t| \leq (\log N)^{-1} \,, \\
C \, |1-e^{it}|^{1/3} \,/\, N^{1/3} & ; \ |t| \geq (\log N)^{-1} \,, \\
\end{cases}
$$
where $C$ is a positive constant which depends only on $\xi$ 
and which may change from occurrence to occurrence in the sequel. 
It follows that
\begin{align*}
\bigg| \int_{\sigma_2} & \frac{h_2(z)}{(1-z)^2} \, \frac{dz}{z} \bigg| = \bigg| \int_{a/N}^{\pi} \frac{h_2(e^{it})}{(1-e^{it})^2} \, dt \bigg| \\
& \leq
  \frac{C}{N^{1/2}} \int_{a/N}^{(\log N)^{-1}} \frac{1}{|1-e^{it}|^{3/2}} \ dt
+ \frac{C}{N^{1/3}} \int_{(\log N)^{-1}}^{\pi} \frac{1}{|1-e^{it}|^{5/3}} \ dz \\
& \leq
  \frac{C}{N^{1/2}} \int_{a/N}^{(\log N)^{-1}} \frac{1}{t^{3/2}} \ dt
+ \frac{C}{N^{1/3}} \int_{(\log N)^{-1}}^{\pi} \frac{1}{t^{5/3}} \ dz \\
& \leq 
  \frac{C}{N^{1/2}} \, (N/a)^{1/2}
+ \frac{C}{N^{1/3}} \, (\log N)^{2/3} \,.
\end{align*}
Since for $a > 1$ sufficiently large, this is clearly $\leq\delta$ for all sufficiently large $N \in \mathbb{N}$,
this proves \eqref{claim13}.

\medskip

The proof of Theorem \ref{bulk-1} is complete now.
\end{proof}


\begin{proof}[Proof of Theorem \ref{edge-1}]
Let $\xi \in \{ \xi^*,\xi_* \}$, $\mu,\nu \in \mathbb{R}$, and $a > 1$.
When we~use the~symbols $\pm$ and $\mp$, we mean 
the upper sign for~$\xi = \xi^*$ and the lower sign for~$\xi = \xi_*$.
We~assume throughout this proof that $N$ is sufficiently large.
Again, by~(\ref{intrep}) and the definition of $Z_N(\xi,\mu,\nu)$, we have
\begin{align}
\label{intrep21}
N^{1/3} \, Z_N(\xi,\pm\mu,\pm\nu) \cdot \frac{f(n,m;\xi N\pm\mu N^{1/3},\xi N\pm\nu N^{1/3})}{n! \, m!} 
= 
\frac{N^{1/3}}{2 \pi i} \int_\sigma h(z) \ \frac{dz}{z} \,,
\end{align}
where
$$
h(z) = h_N(z) := \frac{\exp \left( -(\xi N\pm\tfrac12(\mu+\nu)N^{1/3}) \frac{1+z}{1-z} + b^*z \right) \cdot I_\alpha \left( w(z) \right)}{(1-z)^2 \cdot z^{(n+m)/2}} \,,
$$
$$
w(z) = w_N(z) := 2 \sqrt{\xi^2 N^2 \pm (\mu+\nu) \xi N^{4/3} + \mu \nu N^{2/3}} \cdot \frac{\sqrt{z}}{1-z} \,,
$$
and $\sigma = \sigma_N$ is the contour specified above, with $\eta := 1/3$.
Note that $\sigma$ implicitly depends on $a$.

\pagebreak[2]

For any $a \in [1,\infty]$, let
$$
A(a) := \frac{\exp(b^*)}{4\pi^{3/2}\xi^{1/2}} \int_{-a}^{+a} \frac{\exp \left( ( \tfrac{1}{12} \gamma_\infty (1\!-\!iu)^3 \!-\! \tfrac{1}{2} (\mu\!+\!\nu) \sqrt{\gamma_\infty} (1\!-\!iu) \!-\! \tfrac{1}{4} (\mu\!-\!\nu)^2 (1\!-\!iu)^{-1} ) / \xi \right)}{(1\!-\!iu)^{3/2}} \, du \,.
$$
We will show that for any $\delta > 0$, there exists some $a > 1$ such that the following holds:
\begin{align}
\label{claim21}
\lim_{n \to \infty}
\frac{N^{1/3}}{2 \pi i} \int_{\sigma_{0}} h(z) \ \frac{dz}{z} = A(a) \,,
\end{align}
\begin{align}
\label{claim22}
\limsup_{n \to \infty} \left| N^{1/3} \int_{\sigma_{1}} h(z) \ \frac{dz}{z} \right| \leq \delta \,,
\end{align}
\begin{align}
\label{claim23}
\limsup_{n \to \infty} \left| N^{1/3} \int_{\sigma_{2}} h(z) \ \frac{dz}{z} \right| \leq \delta \,.
\end{align}
Then, by symmetry, similar bounds hold for the integrals along $\sigma_{-1}$ and $\sigma_{-2}$.
Using these results, it is easy to see that
$$
\lim_{N \to \infty}
\frac{1}{2 \pi i} \int_{\sigma} h(z) \ \frac{dz}{z} = A(\infty) \,.
$$
Substituting $z = (\xi/\gamma_\infty)^{1/3} z$, $dz = (\xi/\gamma_\infty)^{1/3} dz$, 
and shifting the path of integration back to the line $1 + i\mathbb{R}$
(which is easily justified by Cauchy's theorem), we obtain
$$
A(\infty) = \frac{\gamma_\infty^{1/6}\exp(b^*)}{4\pi^{3/2}\xi^{2/3} i} \int_{1-i\infty}^{1+i\infty} \frac{\exp \left( \tfrac{1}{12} z^3 - \tfrac{1}{2}(\mu+\nu) \gamma_\infty^{1/6} \xi^{-2/3} z - \tfrac{1}{4}(\mu-\nu)^2 \gamma_\infty^{1/3} \xi^{-4/3} z^{-1} \right)}{z^{3/2}} \, dz \,.
$$
Making the replacements $\mu \mapsto (\xi^{2/3}/\gamma_\infty^{1/6}) \mu$, $\nu \mapsto (\xi^{2/3}/\gamma_\infty^{1/6}) \nu$
and multiplying by~$(\xi^{2/3}/\gamma_\infty^{1/6})$, we further obtain
$$
(\xi^{2/3}/\gamma_\infty^{1/6}) \, A(\infty) = \frac{\exp(b^*)}{4\pi^{3/2}i} \int_{1-i\infty}^{1+i\infty} \frac{\exp \left( \tfrac{1}{12}z^3 - \tfrac{1}{2}(\mu+\nu) - \tfrac{1}{4}(\mu-\nu)^2 z^{-1} \right)}{z^{3/2}} \, dz \,.
$$
By the integral representation for the Airy kernel 
(see \eg Proposition 2.2 in \cite{Ko2}), the latter expression
is equal to $\exp(b^*) \, \mathbb{A}(\mu,\nu)$,
whence \eqref{edge-1F} and \eqref{edge-1G}.

\pagebreak[2]
\medskip

To prove \eqref{claim21} -- \eqref{claim23}, we proceed similarly as in the preceding proof.
To begin with, we write
$$
h(z) = \frac{h_1(z) \, h_2(z)}{(1-z)^2 \, z^{(n+m)/2}} \,,
$$
where
$$
h_1(z) := \exp \left( - \big( \xi N \pm \tfrac12(\mu+\nu)N^{1/3} \big) \frac{1+z}{1-z} + b^*z \right) \cdot \exp \left( + w(z) \right)
$$
and
$$
h_2(z) := \exp \left( - w(z) \right) \cdot I_\alpha \left( w(z) \right) \,.
$$
Similarly as above, we then have
\begin{multline*}
h_1(z)
= 
\exp \bigg( - \xi N \frac{1-\sqrt{z}}{1+\sqrt{z}} \mp \tfrac12(\mu+\nu) N^{1/3} \frac{1-\sqrt{z}}{1+\sqrt{z}} - \tfrac14 (\mu-\nu)^2 \xi^{-1} N^{-1/3} \frac{\sqrt{z}}{1-z} \\ \,+\, b^* z + \frac{\sqrt{z}}{1-z} \, \myo \big( 1/N \big) \bigg) \,.
\end{multline*}
Starting from this representation, the proof is by~and~large similar to the preceding proof. 
However, a notable difference is given by the fact the leading-order terms cancel out now, 
which is why we have to keep track of several additional terms in the asymptotic approximations.

\medskip

We begin with the contour integral for $\sigma_0$.
Let $R = R_N := (1-1/N^{1/3})$.
Substituting $z = Re^{it}$ and $t = u/N^{1/3}$ in the integral
on the left-hand side in~\eqref{claim21}, we obtain
\begin{align}
\label{step21a}
\frac{1}{2 \pi} \int_{-a}^{+a}
\frac{h_1(Re^{iu/N^{1/3}}) \, h_2(Re^{iu/N^{1/3}})}{(1-Re^{iu/N^{1/3}})^2 \cdot (Re^{iu/N^{1/3}})^{(n+m)/2}}
 \ du \,.
\end{align}
Similarly as in the preceding proof, we use the notation $\myo_a$ 
to denote a bound involving an implicit constant depending also on $a$ 
(in addition to $(\alpha_N)$, $\xi,\mu,\nu$).
Then, putting $\hat\gamma_\infty := 1 - \gamma_\infty$ for abbreviation
and using Taylor expansion, we have the approximations, for $|u| \leq a$,
\begin{align*}
h_1(Re^{iu/N^{1/3}})
&= 
\exp \Big( -\tfrac{1}{4} (1-iu) \xi N^{2/3} - \tfrac{1}{8} \xi N^{1/3} - \tfrac{1}{12} \xi + \tfrac{1}{192} (1-iu)^3 \xi \\
&\qquad \mp \tfrac{1}{8} (\mu+\nu) (1-iu) - \tfrac{1}{4} (\mu-\nu)^2 \xi^{-1} (1-iu)^{-1} + b^* \Big) \big( 1+o_a(1) \big) \,,
\\
h_2(Re^{iu/N^{1/3}})
&= 
\frac{1}{\sqrt{4\pi N^{4/3} \xi/(1-iu)}}
\cdot
\exp 
\Big(
    - \tfrac{1}{4} \hat\gamma_\infty^2 (1-iu) \xi^{-1} N^{2/3} - \tfrac{1}{8} \hat\gamma_\infty^2 \xi^{-1} N^{1/3} - \tfrac{1}{12} \hat\gamma_\infty^2 \xi^{-1} \\
    & \qquad - \tfrac{1}{96} \hat\gamma_\infty^2 (1-iu)^3 \xi^{-1} \pm \tfrac{1}{8} \hat\gamma_\infty^2 (\mu+\nu) (1-iu) \xi^{-2} + \tfrac{1}{192} \hat\gamma_\infty^4 (1-iu)^3 \xi^{-3} 
\Big)
\cdot
\big( 1+o_a(1) \big) \,,
\\
   (1-Re^{iu/N^{1/3}})^{2}
&= (1-iu)^2/N^{2/3} \big( 1+o_a(1) \big) \,,
\\
  (Re^{iu/N^{1/3}})^{(m+n)/2}
&= \exp \left( -\tfrac12(1+\gamma_\infty) \big( (1-iu)N^{2/3} + \tfrac12 N^{1/3} + \tfrac13 \big) \right) \big( 1+o_a(1) \big) \,.
\end{align*}

For the second approximation, we have used the uniform asymptotics \eqref{bessel-u1}
for the modified Bessel function. Indeed, since
$$
\alpha_N = \left( 1 - \gamma_\infty \right) N + o(N^{1/3}) \,,
$$
$$
w(Re^{iu/N^{1/3}}) = 2N^{4/3} \xi / (1-iu) \left( 1 + o_a(1) \right) \,,
$$
and
$$
\sqrt{1+z^2} + \log \frac{z}{1+\sqrt{1+z^2}} = z - \tfrac{1}{2} z^{-1} + \tfrac{1}{24} z^{-3} + \myo(z^{-5})
$$
for $|z| \to \infty$, $|\arg z| \leq \tfrac12\pi$, it follows from \eqref{bessel-u1} that
\begin{align*}
   h_2(Re^{iu/N^{1/3}}) 
&= \exp(-w(Re^{iu/N^{1/3}})) \cdot I_\alpha(w(Re^{iu/N^{1/3}})) \\ 
&= \frac{\exp \left( - \tfrac12\alpha^2/w(Re^{iu/N^{1/3}}) + \tfrac{1}{24}\alpha^4/w^3(Re^{iu/N^{1/3}}) \right)}{\sqrt{2\pi w(Re^{iu/N^{1/3}})}} \left( 1 + o_a(1) \right) \,.
\end{align*}
Now, by straightforward expansion,
\begin{align*}
   \frac{1}{w(\gamma(u/N^{1/3}))}
&= \tfrac12 (1-iu) \xi^{-1} N^{-4/3} + \tfrac14 \xi^{-1} N^{-5/3} + \tfrac16 \xi^{-1} N^{-2} \\&\quad\,+\, \tfrac{1}{48} (1-iu)^3 \xi^{-1} N^{-2} \mp \tfrac{1}{4} (\mu+\nu) (1-iu) \xi^{-2} N^{-2} + \myo_a(N^{-7/3}) \,.
\end{align*}
Plugging this into the preceding expression and using once again 
the assumption $\alpha = (1-\gamma_\infty)N + o(N^{1/3})$ yields 
the approximation for $h_2(Re^{iu/N^{1/3}})$.

Inserting the above approximations into \eqref{step21a}
and ordering with respect to fractional powers of $N$, we obtain
$$
\frac{\exp(b^*)}{4 \pi^{3/2}\xi^{1/2}} \int_{-a}^{+a}
\frac{\exp(T)}{(1-iu)^{3/2}} \, (1+o_a(1))
\ dt
$$
with
\begin{align*}
T = 
+ &\Big( - \tfrac14(1-iu)\xi + \tfrac12(1+\gamma_\infty)(1-iu) - \tfrac14\hat\gamma_\infty^2(1-iu)\xi^{-1} \Big) N^{2/3}
\\ 
+ &\Big( - \tfrac18\xi + \tfrac14(1+\gamma_\infty) - \tfrac18\hat\gamma_\infty^2\xi^{-1} \Big) N^{1/3}
\\
+ &\Big( - \tfrac{1}{12}\xi + \tfrac16(1+\gamma_\infty) - \tfrac{1}{12}\hat\gamma_\infty^2\xi^{-1} \Big)
\\
+ &\tfrac{1}{192}(1-iu)^3\xi \mp \tfrac18(\mu+\nu)(1-iu) - \tfrac14(\mu-\nu)^2\xi^{-1}(1-iu)^{-1} \\
- &\tfrac{1}{96}\hat\gamma_\infty^2(1-iu)^3\xi^{-1} \pm \tfrac18\hat\gamma_\infty^2(\mu+\nu)(1-iu)\xi^{-2} + \tfrac{1}{192}\hat\gamma_\infty^4(1-iu)^3\xi^{-3} \,.
\end{align*}
By the characterizing equation \eqref{boundary} for the edge of the spectrum, we have
$$
- \tfrac14 \xi - \tfrac14 (1-\gamma_\infty)^2 / \xi + \tfrac12 (1+\gamma_\infty) = 0 \,.
$$
Thus, the terms in the large round brackets cancel out, and the remaining sum 
can be simplified to
\begin{align*}
T = \left( \tfrac{1}{12} \gamma_\infty (1-iu)^3 \pm \tfrac{1}{4} (\mu+\nu) (1+\gamma_\infty-\xi) (1-iu) - \tfrac{1}{4} (\mu-\nu)^2 (1-iu)^{-1} \right) / \xi \,.
\end{align*}
Recall that $\xi = (1 \pm \sqrt{\gamma_\infty})^2$
at the upper edge and the lower edge of the spectrum, respectively. 
Plugging this into the expression for $T$ and simplifying, we obtain 
\begin{align*}
T = \left( \tfrac{1}{12} \gamma_\infty (1-iu)^3 - \tfrac{1}{2} \sqrt{\gamma_\infty} (\mu+\nu) (1-iu) - \tfrac{1}{4} (\mu-\nu)^2 (1-iu)^{-1} \right) / \xi \,,
\end{align*}
and the proof of \eqref{claim21} is complete.

\pagebreak[2]
\medskip

The proof of \eqref{claim22} is similar to that of \eqref{claim21}.
Setting $z := (1-t/N^{1/3}) \exp(ia/N^{1/3})$ with $t \in [0;1]$
and assuming that $N \in \mathbb{N}$ is large enough,
we~have the following~bounds:
\begin{align*}
|h_1(z)|
&\leq
\Big| \exp \Big( -\tfrac{1}{4} (t-ia) \xi N^{2/3} - \tfrac{1}{8} t^2 \xi N^{1/3} - \tfrac{1}{12} t^3 \xi + \tfrac{1}{192} (t-ia)^3 \xi \\
&\qquad \mp \tfrac{1}{8} (\mu+\nu) (t-ia) - \tfrac{1}{4} (\mu-\nu)^2 \xi^{-1} (t-ia)^{-1} + b^* \Big) \big( 1+o_a(1) \big) \Big| \,, 
\\
|h_2(z)|
&\leq
\Big| \frac{C}{\sqrt{4\pi N^{4/3} \xi/(t-ia)}}
\cdot
\exp 
\Big(
    - \tfrac{1}{4} \hat\gamma_\infty^2 (t-ia) \xi^{-1} N^{2/3} - \tfrac{1}{8} \hat\gamma_\infty^2 t^2 \xi^{-1} N^{1/3} - \tfrac{1}{12} \hat\gamma_\infty^2 t^3 \xi^{-1} \\
    & \qquad - \tfrac{1}{96} \hat\gamma_\infty^2 (t-ia)^3 \xi^{-1} \pm \tfrac{1}{8} \hat\gamma_\infty^2 (\mu+\nu) (t-ia) \xi^{-2} + \tfrac{1}{192} \hat\gamma_\infty^4 (t-ia)^3 \xi^{-3} 
\Big)
\cdot
\big( 1+o_a(1) \big) \Big| \,,
\allowdisplaybreaks
\\
      |1-z|^{2}
&\geq \big| (t-ia)^2/N^{2/3} \big( 1+o_a(1) \big) \big| \,,
\\
      |z|^{(m+n)/2}
&\geq \big| \exp \left( -\tfrac12(1+\gamma_\infty) \big( (t-ia)N^{2/3} + \tfrac12 t^2 N^{1/3} + \tfrac13 t^3 \big) \right) \big( 1+o_a(1) \big) \big| \,.
\end{align*}

Let us comment on the second bound (the others bounds being straightforward).
The uniform asymptotic approximation \eqref{bessel-u1} is not applicable anymore,
since $w(z)$ approaches the imaginary axis as $t \to 0$:
$$
w(z) = 2N^{4/3} \xi / (t-ia) \left( 1 + o_a(1) \right) \,.
$$
However, we may use the uniform asymptotic bound \eqref{bessel-u3} instead,
since $\re \! w(z)$ $\geq 0$, as follows from a similar estimate as in \eqref{bound1c}.
We thus obtain, for all $N \in \mathbb{N}$ sufficiently large (depending on $a$),
\begin{align*}
   \Big| h_2(z) \Big|
&= \Big| \exp(-w(z)) \cdot I_\alpha(w(z)) \Big| \\ 
&\le \Bigg| \frac{C}{\sqrt{2\pi w(z)}} \, \exp \left( - \tfrac12\alpha^2/w(z) + \tfrac{1}{24}\alpha^4/w^3(z) \right) \left( 1 + o_a(1) \right) \Bigg| \,,
\end{align*}
where $C > 0$ is a constant not depending on $a$.
Since
\begin{align*}
   \frac{1}{w(z)}
&= \tfrac12 (t-ia) \xi^{-1} N^{-4/3} + \tfrac14 t^2 \xi^{-1} N^{-5/3} + \tfrac16 t^3 \xi^{-1} N^{-2} \\&\quad\,+\, \tfrac{1}{48} (t-ia)^3 \xi^{-1} N^{-2} \mp \tfrac{1}{4} (\mu+\nu) (t-ia) \xi^{-2} N^{-2} + \myo_a(N^{-7/3}) \,,
\end{align*}
this entails the asserted bound by substituting $1/w(z)$ and simplifying.

Putting it all together and doing similar simplifications as in the proof of \eqref{claim21}, 
we~find that the integral on the left-hand side in \eqref{claim22} is bounded above by
$$
\frac{C \exp(b^*)}{\sqrt{4\pi\xi}} \int_{0}^{1}
\frac{|\exp(T)|}{|t-ia|^{3/2}} \, (1+o_a(1))
\ dt \,,
$$
where $C$ is a positive constant which does not depend on $a$
(and which may change from step to step as usual)
and
\begin{align*}
T = \left( \tfrac{1}{12} \gamma_\infty (t-ia)^3 - \tfrac{1}{2} (\mu+\nu) \sqrt{\gamma_\infty} (t-ia) - \tfrac{1}{4} (\mu-\nu)^2 (t-ia)^{-1} \right) / \xi \,.
\end{align*}
The real part of this expression is obviously of order $\mathcal{O}(1)$.
It therefore follows that
$$
\limsup_{N \to \infty} \frac{C}{\sqrt{4\pi\xi}} \int_{0}^{1} \frac{|\exp(T)|}{|t-ia|^{3/2}} \, (1+o_a(1)) \ dt
\leq
\frac{C a^{-3/2}}{\sqrt{4\pi\xi}} \,.
$$
Since for $a > 1$ sufficiently large, the right-hand side is clearly
bounded above by~$\delta$, this proves \eqref{claim22}.

\pagebreak[2]

It remains to prove \eqref{claim23}. 
Using similar estimates as in \eqref{bound1a} and \eqref{bound1b},
we see that it suffices to show that
\begin{align}
\label{claim24}
\limsup_{N \to \infty} \left( N^{1/3} \int_{a/N^{1/3}}^{\pi} \frac{|h_2(e^{it})|}{|1-e^{it}|^2} \, dt \right) \leq \delta
\end{align}
for $a > 1$ sufficiently large. 
Moreover, by similar estimates as in \eqref{bound1c} and \eqref{bound1d},
we have $\re w(e^{it}) \geq 0$ and $\im w(e^{it}) \geq 0$ for all $t \in [a/N^{1/3};\pi]$, 
so we can use \eqref{jbound-1a}~and~\eqref{jbound-2a} to bound $I_\alpha(w(e^{it}))$.

Let $\varepsilon \in (0;1/3)$ be a small constant which will be chosen later,
let~$I$ denote the~subset of those $t \in [a/N^{1/3};\pi]$
such that $\im w(e^{it}) / \alpha_N \in (1-\varepsilon;1+\varepsilon)$,
and let~$J$ denote the complement of this subset.
Using the stronger bound \eqref{jbound-2a} on~$J$
and the weaker bound \eqref{jbound-1a} on~$I$,
we then have, for sufficiently large $N \in \mathbb{N}$,
\begin{align*}
     \int_{a/N^{1/3}}^{\pi} \frac{|h_2(e^{it})|}{|1-e^{it}|^2} dt
&\leq C_\varepsilon N^{-1/2} \int_{J} \frac{1}{|1-e^{it}|^{3/2}} \, dt + C N^{-1/3} \int_{I} \frac{1}{|1-e^{it}|^{5/3}} \, dt \\
&\leq C_\varepsilon N^{-1/2} \int_{J} t^{-3/2} \, dt + C N^{-1/3} \int_{I} t^{-5/3} \, dt \\
&\leq C_\varepsilon N^{-1/2} (N^{1/3}/a)^{1/2} + C N^{-1/3} \mylebesgue(I) / (\inf I)^{5/3} \,.
\end{align*}
Here $\mylebesgue(I)$ denotes the length of the interval $I$,
$C$ is a constant which depends only on $\xi$,
and $C_\varepsilon$ is a constant which may additionally depend on $\varepsilon$.
(Of course, both constants may change from occurrence to occurrence as usual.)
Clearly, once $\varepsilon > 0$ is fixed, we may make the first term 
arbitrarily small by choosing $a > 0$ sufficiently~large.
Hence, to complete the proof of \eqref{claim24}, it remains to establish
an appropriate bound on the second term, \ie on~$\mylebesgue(I) / (\inf I)^{5/3}$.
To~this purpose, first note that
\begin{align*}
   \im w(e^{it}) 
&= N\xi (\sin (t/2))^{-1} \, (1+o(1))
\end{align*}
and therefore
$$
\im w(e^{it}) / \alpha_N = \frac{\xi}{1-\gamma_N} (\sin (t/2))^{-1} \, (1+o(1)) \,.
$$
For the upper edge of the spectrum, we have $\xi > 1$,
which implies that the interval $I$ is empty 
for $\varepsilon$ sufficiently small and $N$ sufficiently large,
and \eqref{claim24} is proven. \linebreak
For the lower edge of the spectrum, we have $0 < \xi < 1$
as well as $\xi\,/\,(1-\gamma_\infty) < 1$ (as is readily verified)
and therefore, for $\varepsilon$ sufficiently small and $N$ sufficiently large,
\begin{align*}
I \subset \bigg[ a/N^{1/3};\pi \bigg] \cap 2 \arcsin \left( \frac{\xi}{(1-\gamma_\infty)(1+2\varepsilon)};\frac{\xi}{(1-\gamma_\infty)(1-2\varepsilon)} \right) \,.
\end{align*}
In particular, $\inf I \geq c$ for some positive constant $c$ depending only 
on $\xi$ and $\gamma_\infty$. Moreover, $\mylebesgue(I)$ can be made arbitrarily small 
by picking $\varepsilon$ sufficiently close to zero. 
Thus, \eqref{claim24} is also proven.

\medskip

The proof of Theorem 1.3 is complete now.
\end{proof}


\newpage

\section{Some bounds for Bessel functions}

In this section we state some uniform asymptotic approximations for Bessel functions
from the literature (see Chapters 10~and~11 in \textsc{Olver} \cite{O}).
Moreover, we deduce several asymptotic bounds for Bessel functions
which are sufficient for controlling the error bounds in the preceding proofs.
These asymptotic bounds should be well-known, but we have not been able 
to find explicit statements suiting our purposes in the literature.
Throughout this section, we~assume that $\alpha \in \mathbb{N}$
and $\alpha \to \infty$.

We need the following uniform asymptotic approximations for Bessel functions,
which can be extracted from Sections 10.7 and 11.10 in \textsc{Olver} \cite{O},
respectively:

\begin{proposition}
For any $\varepsilon > 0$,
\begin{align}
\label{bessel-u1}
I_\alpha(\alpha z) = \frac{1}{\sqrt{2\pi \alpha}} \frac{\exp(\alpha \eta)}{(1+z^2)^{1/4}} \left( 1 + \myo_{\varepsilon}(1/\alpha) \right) \qquad (\alpha \to \infty) \,,
\end{align}
where 
\begin{align}
\label{etadef}
\eta := \sqrt{1+z^2} + \log \frac{z}{1+\sqrt{1+z^2}}
\end{align}
and the $\myo_\varepsilon$-bound holds uniformly in the set
$
\left\{ z \in \mathbb{C} : |\arg z| \leq \tfrac12\pi - \varepsilon \right\} .
$
\end{proposition}

\begin{proposition}
For any $\varepsilon > 0$,
\begin{multline}
\label{bessel-u2}
J_\alpha(\alpha z) = \left( \frac{4\zeta}{1-z^2} \right)^{1/4} \bigg\{ \frac{\ai(\alpha^{2/3} \zeta)}{\alpha^{1/3}} \big( 1 + \myo_\varepsilon(1/\alpha) \big) \\ + \frac{\ai'(\alpha^{2/3} \zeta)}{\alpha^{5/3}} \myo_\varepsilon \big( (1+|\zeta|^{1/2})^{-1} \big) \bigg\} \qquad (\alpha \to \infty) \,,
\end{multline}
where
\begin{align}
\label{zetadef}
\tfrac23\zeta^{3/2} := \log \frac{1+\sqrt{1-z^2}}{z} - \sqrt{1-z^2}
\end{align}
and the $\myo_\varepsilon$-bounds hold uniformly in the set
$
\left\{ z \in \mathbb{C} : |\arg z| \leq \pi - \varepsilon \right\} .
$
\end{proposition}

We refer to Sections 10.7 and 11.10 in \textsc{Olver} \cite{O}
for a discussion of the choices of the various branches.

By Proposition 3.1 we have the following bound on $I_\alpha(\alpha z)$,
uniformly in the set 
$\{ z \in \mathbb{C} : |z| \geq 1+\varepsilon \,\wedge\, |\arg z| \leq \tfrac12\pi - \varepsilon \}$:
$$
I_\alpha(\alpha z) = \myo_\varepsilon \left( \frac{\exp(\alpha \eta)}{(\alpha z)^{1/2}} \right) \qquad (\alpha \to \infty) \,.
$$
Using Proposition 3.2 it can be shown that this bound in fact remains valid 
up to the imaginary axis:

\begin{lemma}
For any $\varepsilon > 0$, 
\begin{align}
\label{bessel-u3}
I_\alpha(\alpha z) = \mathcal{O}_\varepsilon \left( \frac{\exp(\alpha\eta)}{(\alpha z)^{1/2}} \right) \qquad (\alpha \to \infty) \,,
\end{align}
where $\eta$ is defined as in \eqref{etadef}
and the $\myo_\varepsilon$-bound holds uniformly in the set \linebreak
$\{ z \in \mathbb{C} : |z| \geq 1+\varepsilon \,\wedge\, |\arg z| \leq \tfrac12\pi \}$.
\end{lemma}

\pagebreak[2]

\begin{proof}
Fix $\varepsilon > 0$, and consider $J_\alpha(\alpha z)$ with $\alpha \in \mathbb{N}$, 
$z \in M_\varepsilon := \{ z \in \mathbb{C} : |z| \geq 1+\varepsilon$ $\wedge$ $|\arg z| \leq \pi - \varepsilon \}$.
We then have the uniform asymptotic approximation \eqref{bessel-u2}.
We~now use the following asymptotic approximations for the Airy function
and its derivative, which hold for $|z| \to \infty$, $|\arg z| \leq \tfrac23\pi - \varepsilon'$
(see Chapters 11.1~and~11.8 in \textsc{Olver} \cite{O}):
\begin{multline*}
\ai\mskip5mu(-z) = \pi^{-1/2} \, z^{-1/4} \bigg\{ \cos(\tfrac23 z^{3/2}-\tfrac14\pi) \left( 1 + \myo_{\varepsilon'}(z^{-3/2}) \right) \\ + \sin(\tfrac23 z^{3/2}-\tfrac14\pi) z^{-3/2} \left( 1 + \myo_{\varepsilon'}(z^{-3/2}) \right) \bigg\} \,,
\end{multline*}
\begin{multline*}
\ai      '  (-z) = \pi^{-1/2} \, z^{+1/4} \bigg\{ \sin(\tfrac23 z^{3/2}-\tfrac14\pi) \left( 1 + \myo_{\varepsilon'}(z^{-3/2}) \right) \\ - \cos(\tfrac23 z^{3/2}-\tfrac14\pi) z^{-3/2} \left( 1 + \myo_{\varepsilon'}(z^{-3/2}) \right) \bigg\} \,.
\end{multline*}
Using the inequalities
$|\cos z| \leq \exp(|\im z|)$, 
$|\sin z| \leq \exp(|\im z|)$,
and the fact that $\ai$, $\ai'$ are holomorphic functions,
it follows that
$$
\ai\mskip5mu(-z) = \myo_{\varepsilon'} \left( z^{-1/4} \, \exp(|\im \tfrac23 z^{3/2}|) \right) \,,
$$
$$
\ai      '  (-z) = \myo_{\varepsilon'} \left( z^{+1/4} \, \exp(|\im \tfrac23 z^{3/2}|) \right) \,,
$$
uniformly in $z \in \mathbb{C}$ such that $|z| \geq \varepsilon'$ and $|\arg z| \leq \tfrac23\pi - \varepsilon'$.

It can be checked that for $\alpha \in \mathbb{N}$, $z \in M_\varepsilon$, 
\mbox{$|-\alpha^{2/3}\zeta| \geq \varepsilon'$} and 
\mbox{$\arg (-\alpha^{2/3}\zeta) \leq \frac{2}{3} \pi - \varepsilon'$}
for some $\varepsilon' > 0$ depending only on $\varepsilon > 0$.
Hence,
$$
\ai\mskip5mu(\alpha^{2/3} \zeta) = \myo_\varepsilon \left( \alpha^{-1/6} (-\zeta)^{-1/4} \exp(|\im i\alpha\xi|) \right) \,,
$$
$$
\ai      '  (\alpha^{2/3} \zeta) = \myo_\varepsilon \left( \alpha^{+1/6} (-\zeta)^{+1/4} \exp(|\im i\alpha\xi|) \right) \,,
$$
uniformly in $z \in M_\varepsilon$,
where $\zeta = \zeta(z)$ is defined as in \eqref{zetadef}
and $\xi = \xi(z) := \tfrac{2}{3} \zeta^{3/2}$.
Inserting this into \eqref{bessel-u2},
it~follows that
\begin{align*}
J_\alpha(\alpha z) = \myo_\varepsilon \left( \frac{\exp(|\re \alpha\xi(z)|)}{(\alpha z)^{1/2}} \right) \qquad (\alpha \to \infty) \,,
\end{align*}
the $\myo_\varepsilon$-bound holding uniformly in $z \in M_\varepsilon$.

Thus, for $z \in \{ z \in \mathbb{C} : |z| \geq 1+\varepsilon \,\wedge\, |\arg z| \leq \tfrac12\pi \}$,
we have
$$
|I_\alpha(\alpha z)| = |J_\alpha(-i\alpha z)| = \myo_\varepsilon \left( \frac{\exp(|\re \alpha\xi(-iz)|)}{(\alpha z)^{1/2}} \right) \qquad (\alpha \to \infty) \,,
$$
when $\im z \geq 0$ and
$$
|I_\alpha(\alpha z)| = |J_\alpha(+i\alpha z)| = \myo_\varepsilon \left( \frac{\exp(|\re \alpha\xi(+iz)|)}{(\alpha z)^{1/2}} \right) \qquad (\alpha \to \infty) \,,
$$
when $\im z \leq 0$.
A comparison of the definitions \eqref{etadef} and \eqref{zetadef}
(together with a~careful discussion of the choice of branches)
shows that
$
\xi(\mp iz) = - \eta(z) \pm \tfrac12 \pi i \,,
$
which completes the proof of the lemma.
\end{proof}

\pagebreak[2]

The next lemma gives some upper bounds for the (unmodified) Bessel function
on the positive real half-axis:

\begin{lemma}
\label{jboundlemma}
There exists a constant $C > 0$ 
such that for all sufficiently large $\alpha \in \mathbb{N}$,
\begin{align}
\label{jbound-1}
|J_\alpha(\alpha x)| \leq \frac{C}{|\alpha x|^{1/3}} 
\qquad&\text{for all $x \in (0,\infty)$} \,.
\end{align}
More precisely, for any $\varepsilon > 0$, there exists a constant $C_\varepsilon > 0$ 
such that for all sufficiently large $\alpha \in \mathbb{N}$,
\begin{align}
\label{jbound-2}
|J_\alpha(\alpha x)| \leq \frac{C_\varepsilon}{|\alpha x|^{1/2}} 
\qquad&\text{for all $x \in (0,1-\varepsilon) \cup (1+\varepsilon,\infty)$} \,.
\end{align}
\end{lemma}

These bounds are not the best possible, but they are sufficient for our purposes.
In~terms of the modified Bessel function, they read
\begin{align}
\label{jbound-1a}
|I_\alpha(i \alpha x)| \leq \frac{C}{|\alpha x|^{1/3}} 
\qquad&\text{for all $x \in (0,\infty)$}
\end{align}
and
\begin{align}
\label{jbound-2a}
|I_\alpha(i \alpha x)| \leq \frac{C_\varepsilon}{|\alpha x|^{1/2}} 
\qquad&\text{for all $x \in (0,1-\varepsilon) \cup (1+\varepsilon,\infty)$} \,,
\end{align}
respectively. This is the form in which they have been used in the last section.

\begin{proof}[Proof of Lemma \ref{jboundlemma}]
We start from the uniform asymptotic approximation \eqref{bessel-u2},
which we now consider for $x > 0$ only.
Note that $\zeta = \zeta(x)$ is a decreasing function of $x$
with $\zeta > 0$ for $x < 1$ and $\zeta < 0$ for $x > 1$.
There exists a constant $C' > 0$ such~that 
\begin{align}
\label{jbound-tmp}
|J_\alpha(\alpha x)| \leq C' \left| \frac{4\zeta}{1-x^2} \right|^{1/4} \left\{ \frac{|\ai(\alpha^{2/3} \zeta)|}{\alpha^{1/3}} + \frac{|\ai'(\alpha^{2/3} \zeta)|}{\alpha^{5/3} (1 + |\zeta|^{1/2})} \right\}
\end{align}
for all $x > 0$ for all sufficiently large $\alpha \in \mathbb{N}$.
We now use the bounds 
\begin{align*}
\ai \mskip5mu (+x) &= \myo\big( x^{-1/4} \, \exp(-\tfrac23x^{3/2}) \big) && (x \to +\infty) \,,
\\
\ai'(+x) &= \myo\big( x^{+1/4} \, \exp(-\tfrac23x^{3/2}) \big) && (x \to +\infty) \,,
\\
\ai \mskip5mu (-x) &= \myo\big( x^{-1/4} \big) && (x \to +\infty) \,,
\\
\ai'(-x) &= \myo\big( x^{+1/4} \big) && (x \to +\infty) \,,
\end{align*}
which follow from well-known asymptotic approximations 
for the Airy function and its~derivative
(see \eg Section~11.1 in \textsc{Olver} \cite{O}). 
In particular, these bounds imply that 
$\ai(x) = \myo((1+|x|^{1/4})^{-1})$ and $\ai'(x) = \myo(1 + |x|^{1/4})$ 
throughout the~real~line. It therefore follows from \eqref{jbound-tmp} 
that there exist constants $C'',C''' > 0$ such that
\begin{align*}
      |J_\alpha(\alpha x)| 
& \leq C'' \left| \frac{4\zeta}{1-x^2} \right|^{1/4} \left\{ \frac{1}{\alpha^{1/3} (1+\alpha^{1/6}|\zeta|^{1/4})} + \frac{(1+\alpha^{1/6}|\zeta|^{1/4})}{\alpha^{4/3} (1 + \alpha^{1/3} |\zeta|^{1/2})} \right\} \\
&\leq C''' \min \left\{ \left| \frac{\zeta}{1-x^2} \right|^{1/4} \, \frac{1}{\alpha^{1/3}} , \left| \frac{1}{1-x^2} \right|^{1/4} \, \frac{1}{\alpha^{1/2}} \right\} 
\end{align*}
for all $x > 0$ for all sufficiently large $\alpha \in \mathbb{N}$.

To deduce \eqref{jbound-1}, use the first term inside the minimum 
and observe that since
$|\zeta| \sim \tfrac{3}{2} |x|^{2/3}$ for $x \to \infty$, 
$|\zeta| \sim \tfrac{3}{2} |\log x|^{2/3}$ for $x \to 0$, 
and $\zeta$ is an analytic function of $x \in (0,\infty)$
with $\zeta(1) = 0$ (see \eg Section~11.10 in \textsc{Olver} \cite{O}), 
$$ 
\left( \frac{4\zeta}{1-x^2} \right)^{1/4} = \myo(x^{-1/3}) \qquad (0 < x < \infty) \,.
$$

To deduce \eqref{jbound-2}, use the second term inside the minimum
and observe that for~$x \not\in (1-\varepsilon,1+\varepsilon)$,
$|1-x^2|^{-1/4} = \myo_\varepsilon(x^{-1/2})$.
\end{proof}


\newpage

\enlargethispage{6.0\baselineskip}


\begin{thebibliography}{?????}

    \bibitem[AS]{AS}
    Abramowitz, M.; Stegun, I. (1965):
    \textit{Handbook of Mathematical Functions.}
    Dover Publications, New~York.    

    \bibitem[AF]{AF3}
    Akemann, G.; Fyodorov, Y.V. (2003):
    Universal random matrix correlations of ratios of characteristic polynomials at the spectral edges.
    \emph{Nuclear Physics B}, \textbf{664}, 457--476.

	\bibitem[BDS]{BDS}
	Baik, J.; Deift, P.; Strahov, E. (2003):
	Products and ratios of characteristic polynomials of random hermitian matrices.
	\textit{J. Math. Phys.}, \textbf{44}, 3657--3670.

    \bibitem[BP]{BP}
    Ben Arous, G.; P\'ech\'e, S. (2005):
	Universality of local eigenvalue statistics for some sample covariance matrices.
	\textit{Comm. Pure Appl. Math.}, \textbf{58}, 1316--1357.

	\bibitem[BS]{BS}
	Borodin, A.; Strahov, E. (2006):
	Averages of characteristic poly\-nomials in random matrix theory.
	\textit{Comm. Pure Appl. Math.}, \textbf{59}, 161--253.

    \bibitem[BH1]{BH1}
    Br\'ezin, E.; Hikami, S. (2000):
	Characteristic polynomials of random matrices.
	\textit{Comm. Math. Phys.}, \textbf{214}, 111--135.

	\bibitem[BH2]{BH2}
	Br\'ezin, E.; Hikami, S. (2001):
	Characteristic polynomials of real symmetric random matrices.
	\textit{Comm. Math. Phys.}, \textbf{223}, 363--382.

    \bibitem[BH3]{BH3}
	Br\'ezin, E.; Hikami, S. (2003):
	New correlation functions for random matrices and integrals over supergroups.
	\textit{J. Phys. A: Math. Gen.}, \textbf{36}, 711--751.

 	\bibitem[De]{De}
 	Deift, P.A. (1999):
 	\emph{Orthogonal Polynomials and Random Matrices: A Riemann-Hilbert Approach.}
 	Courant Lecture Notes in Mathematics, vol. 3, Courant Institute of Mathematical Sciences, New York.

  	\bibitem[Er]{BMP}
  	Erd\'elyi, A.; Magnus, W.; Oberhettinger, F.; Tricomi, F.G. (1954):
  	\emph{Tables of Integral Transforms}, volume I.
  	McGraw-Hill Book Company, New York.

    \bibitem[FS-1]{FeSo}
    Feldheim, O.N.; Sodin, S. (2008):
    A universality result for the smallest eigenvalues 
        of certain sample covariance matrices.
    Preprint.

    \bibitem[Fo]{Fo}
    Forrester, P.J. (2008+):
    \textit{Log Gases and Random Matrices}.
    Book in preparation, \verb|www.ms.unimelb.edu.au/~matpjf/matpjf.html|

    \bibitem[FS-2]{FS3}
    Fyodorov, Y.V.; Strahov, E (2003):
    An exact formula for general spectral correlation function
        of random Hermitian matrices.
    \textit{J. Phys. A: Math. Gen.}, \textbf{36}, 3202--3213. 

    \bibitem[GK]{GK}
    G\"otze, F.; K\"osters, H. (2009):
    On the second-order correlation function of the characteristic polynomial
        of a Hermitian Wigner matrix.
    \textit{Comm. Math. Phys.}, \textbf{285}, 1183--1205.

    \bibitem[K\"o1]{Ko1}
    K\"osters, H. (2008):
    On the second-order correlation function of the characteristic polynomial
        of a real-symmetric Wigner matrix.
    \textit{Electron. Comm. Prob.}, \textbf{13}, 435--447.

    \bibitem[K\"o2]{Ko2}
    K\"osters, H. (2008):
    Asymptotics of characteristic polynomials of Wigner matrices at the edge of the spectrum.
    Preprint.

    \bibitem[K\"o3]{Ko3}
    K\"osters, H. (2009):
    Characteristic Polynomials of Sample Covariance Matrices.
    Preprint.

     \bibitem[Me]{Me}
     Mehta, M.L. (2004):
     \textit{Random Matrices}, 3rd edition.
     Pure and Applied Mathematics, vol. 142, Elsevier, Amsterdam.

	\bibitem[Ol]{O}
	Olver, F.W.J. (1974):
	\textit{Asymptotics and Special Functions.}
	Academic Press, New~York.

    \bibitem[P\'e]{Pe}
    P\'ech\'e, S. (2009):
    Universality results for the largest eigenvalues of some sample covariance matrix ensembles.
    \textit{Prob. Theory Rel. Fields}, \textbf{143}, 481--516.

    \bibitem[So]{So}
    Soshnikov, A. (2002):
	A note on universality of the distribution of the largest eigenvalues
		in certain sample covariance matrices.
	\textit{J. Stat. Phys.}, \textbf{108}, 1033--1056.

    \bibitem[SF]{SF3}
    Strahov, E.; Fyodorov, Y.V. (2003):
    Universal results for correlations of characteristic polynomials:
        Riemann-Hilbert approach.
    \textit{Comm. Math. Phys.}, \textbf{241}, 343--382.

	\bibitem[Sz]{Sz}
	Szeg\"o, G. (1967):
	\textit{Orthogonal Polynomials}, 3rd  edition.
	American Mathematical Society Colloquium Publications, vol. XXIII, American Mathematical Society, Providence, Rhode Island.

    \bibitem[TV]{TV5}
    Tao, T.; Vu, V. (2009):
	Random Covariance Matrices: Universality of Local Statistics of Eigenvalues.
	Preprint.

	\bibitem[Va]{Va}
	Vanlessen, M. (2003):
	Universal Behavior for Averages of Characteristic Polynomials
		at the Origin of the Spectrum.
	\textit{Comm. Math. Phys.}, \textbf{253}, 535--560.

\end{thebibliography}
\end{document}